\documentclass[hidelinks,onefignum,onetabnum]{siamart251104}

\usepackage[dvipsnames]{xcolor}
\definecolor{linkcolor}{RGB}{74, 102, 146}
\usepackage{hyperref}
\usepackage{cleveref}
\usepackage{url}
\usepackage{amsmath, amssymb, amsfonts}
\usepackage{graphicx}
\usepackage{comment}
\usepackage{subcaption}
\usepackage{booktabs}
\usepackage{wrapfig}
\usepackage{multirow}
\usepackage{multicol}
\usepackage{lipsum}
\usepackage{epstopdf}
\usepackage{algorithmic}
\ifpdf
  \DeclareGraphicsExtensions{.eps,.pdf,.png,.jpg}
\else
  \DeclareGraphicsExtensions{.eps}
\fi

\newsiamremark{remark}{Remark}
\newsiamremark{hypothesis}{Hypothesis}
\crefname{hypothesis}{Hypothesis}{Hypotheses}
\newsiamthm{claim}{Claim}
\newsiamremark{fact}{Fact}
\crefname{fact}{Fact}{Facts}

\headers{Spectrally Safe Neural Operator}{J.~Oh, Y.~Lee, J.~Darbon, and G.~E.~Karniadakis}

\title{Spectrally Safe Neural Operator Warm-Starts for Large-Scale Newton Solvers}

\author{
Jaemin Oh\thanks{Division of Applied Mathematics, Brown University, Providence, RI 02912, USA (\email{jaeminoh.math@gmail.com}).}
\and
Youngkyu Lee\footnotemark[1]
\and
J\'er\^ome Darbon\footnotemark[1]
\and
George Em Karniadakis\footnotemark[1]
}

\usepackage{amsopn}

\ifpdf
\hypersetup{
  pdftitle={Hybrid Newton Failure Modes},
  pdfauthor={Jaemin Oh}
}
\fi

\begin{document}

\maketitle

\begin{abstract}
Neural operators are increasingly used to warm-start Newton solvers for nonlinear PDEs, on the premise that a low test error places the initial guess inside the basin of attraction. We show that this premise is unreliable. 
An operator trained to the relative \(L^2\) error \(O(10^{-3})\) can still produce an initial state in which the discrete Jacobian is indefinite, because the mean-squared training controls error on average while leaving localized pointwise violations of the underlying physics. 
For a nearly incompressible hyperelasticity problem, we trace this to the predicted volume change: the operator disperses \(\mathrm{det} F\) well away from one, and the resulting Jacobian acquires negative eigenvalues even when the predicted field is visually indistinguishable from the reference. At a small scale, this is a nuisance; at a multi-million degree-of-freedom scale, it is disqualifying, since the conjugate gradient and other Krylov solvers needed for memory-feasible Newton steps assume a definite spectrum.
We then show that a short, label-free fine-tuning phase -- penalizing the operator against the discrete energy, with no additional solution data -- shifts the Jacobian spectrum back to positive definite. Combined with an inexact outer loop, this gives a warm-started Newton method that converges across the full loading range where the unregularized operator fails, reaching up to 5.4\(\times\) wall-clock speedup over incremental continuation on a 3D problem with 6.4 million degrees of freedom.
\end{abstract}



\section{Introduction}
Since the advent of neural operators~\cite{lu2021learning, li2021fourier}, data-driven surrogates have emerged as a powerful paradigm for numerically solving partial differential equations (PDEs). In many-query settings -- where traditional solvers must be executed repeatedly under varying parameters, boundary conditions, or source terms -- classical numerical methods incur prohibitive computational overhead. Conversely, once trained across diverse parametric spaces, boundary conditions, and geometric configurations, neural operators can predict PDE solutions for unseen configurations rapidly~\cite{lee2025finite,choi2024spectral}. This exceptional evaluation speed has catalyzed their adoption in large-scale simulations, engineering design, and many-query optimization tasks~\cite{bonev2023spherical, wu2024transolver, lu2022multifidelity}. Nevertheless, a primary bottleneck of purely data-driven surrogates remains their limited accuracy~\cite{bacho2025operator}.

In principle, universal approximation theorems guarantee the existence of a neural operator capable of mapping underlying solution operators with arbitrary precision~\cite{chen1995universal,boulle2023elliptic}. In practice, however, sampling errors during dataset generation and non-convex optimization hurdles during training frequently prevent these models from reaching their theoretical limits~\cite{lee2024training,hong2026error}. To bridge this gap, hybrid frameworks combining neural operators with classical numerical methods have emerged as a promising paradigm~\cite{zhang2022blending, kopanivcakova2025deeponet, eshaghi2026nows}. By pairing the rapid inference of data-driven surrogates with the convergence of classical solvers, these hybrid methodologies accelerate computation while strictly preserving exactness.

For linear problems, neural operators can be used as preconditioners for iterative methods. Motivated by the observation that neural networks inherently resolve low-frequency components first -- a phenomenon known as spectral bias~\cite{rahaman2019spectral} -- Zhang et al.~\cite{zhang2022blending} proposed blending neural operators with classical relaxation methods (e.g., Jacobi or Gauss--Seidel) that efficiently eliminate high-frequency errors. Similarly, Kopanivcakova and Karniadakis~\cite{kopanivcakova2025deeponet} leveraged the trunk networks of deep operator networks~\cite{lu2021learning} as localized basis functions to construct prolongation and restriction operators within the subspace correction framework~\cite{xu1992iterative}. Taking a direct initialization approach, Eshaghi et al.~\cite{eshaghi2026nows} initialized standard Krylov-subspace iterative solvers with neural operator predictions to reduce the iteration count required for convergence.

Extending these hybrid methods to nonlinear problems requires coupling neural operators with Newton-based iterative schemes. Although the classical Newton method exhibits quadratic convergence, this rapid behavior is local and lacks global guarantees; convergence is achieved only if the initial guess resides within the local basin of attraction. Because trained neural operators provide high-quality approximations to solutions, they can serve as excellent initial guesses that reliably place the system within this basin. Consequently, several neural operator-based warm-start Newton frameworks (WS-Newton) have recently emerged. In solid mechanics, the Neural-Initialized Newton framework proposed by Taghikhani et al.~\cite{taghikhani2025neural} learns parametric solution spaces to warm-start finite element simulations. Similarly, Wang et al.~\cite{wang2026pretrain, wang2026pretraining} utilized physics-informed and data-driven training stages to construct initial solutions, which are subsequently refined via classical finite element routines. Alternatively, Lee et al.~\cite{lee2025neural} proposed a neural operator-based nonlinear preconditioning strategy (NP-Newton) where a trained neural operator repeatedly corrects intermediate iterates, thereby helping the solver escape line-search traps and severe unbalanced nonlinearities~\cite{lanzkron1996analysis}. Beyond the Newton method, Zhou et al.~\cite{zhou2025neural} demonstrated that neural operator-based super-fidelity schemes can successfully warm-start steady-state fluid flow simulations.

Despite these foundational contributions, the existing literature -- including WS-Newton and NP-Newton -- lacks a systematic exposition of the challenges and failure modes that arise when scaling neural operator-based hybrid Newton methods to high-dimensional, multi-million degree-of-freedom (DOF) nonlinear 3D systems.

To address this critical gap, this work systematically identifies the fundamental failure modes inherent to neural operator-based hybrid Newton methods. Importantly, we demonstrate that hybridization can fail even when the underlying neural operator achieves an excellent average relative $L^2$ error of $O(10^{-3})$ on the test dataset. We show that this hybridization failure stems from the objective functions of data-driven neural operators, which typically minimize errors in Euclidean spaces, such as the mean-squared error. Consequently, while their predictions may appear highly accurate globally, they frequently exhibit severe, localized violations of physical constraints. In nonlinear regimes, these localized errors dramatically degrade the condition number or induce spurious negative eigenvalues in the discrete Jacobian matrices. This phenomenon fundamentally exacerbates unbalanced nonlinearities, trapping the standard Newton loop in local minima or line-search stagnation points.

To overcome these vulnerabilities, we introduce a streamlined, dual-stage computational framework. We first pre-train a neural operator in a data-driven manner to capture large-scale features. This is followed by a targeted, physics-informed fine-tuning stage. By explicitly regularizing the neural operator against the discrete algebraic operator, our approach suppresses localized errors and ensures compatibility with the classical solver. This strategy guarantees stable convergence where WS-Newton fails, achieving up to 
5.4× wall-clock speedup over incremental continuation on a 3D problem with 6.4 million degrees of freedom.

The remainder of this paper is organized as follows. \Cref{sec:preliminaries} revisits Newton's method and neural operator-based hybrid formulations. \Cref{sec:examples} details the potential failure modes inherent to these hybrid frameworks through four examples. Finally, \Cref{sec:conclusion} discusses our findings and concludes the work.

\section{Preliminaries}\label{sec:preliminaries}
\subsection{Newton's method}
Consider a nonlinear algebraic system resulting from the spatial discretization of a governing PDE, expressed as:
\[
    r(u) = 0,
\]
where $u \in \mathbb{R}^N$ represents the state vector of DOFs, and $r: \mathbb{R}^N \to \mathbb{R}^N$ denotes the discrete nonlinear residual vector. The classical Newton--Raphson method iteratively updates an initial state estimate $u_0$ via the following linearized sequence:
\begin{equation}\label{eq:newton}
    J(u_k) \delta u_k = -r(u_k), \quad u_{k+1} = u_k + \eta_k \delta u_k,
\end{equation}
where $J(u_k) = \frac{\partial r}{\partial u} \big|_{u_k} \in \mathbb{R}^{N \times N}$ is the system Jacobian matrix, $\delta u_k$ is the search direction, and $\eta_k \in (0, 1]$ is a scalar step length determined via globalization strategies such as line-search or trust-region frameworks~\cite{nocedal1999numerical}.

When the initial guess $u_0$ is chosen within the local basin of attraction, the sequence exhibits a quadratic rate of asymptotic convergence~\cite{kantorovich1948functional}. However, for severely nonlinear problems, the radius of this local convergence basin shrinks considerably. If $u_0$ resides outside this critical boundary, the linear corrections become highly susceptible to stagnation, line-search step-size collapse, or catastrophic divergence. For such ill-conditioned cases, a continuation method is conventionally employed to stabilize convergence.

\subsubsection{Newton continuation method}\label{sec:ic-newton}
Suppose the nonlinear residual vector is parameterized by a scalar or vector $\lambda \in \mathbb{R}$, such that our target problem is given by:
\[
    r(u; \lambda^*) = 0,
\]
where $\lambda^*$ denotes a parameter value (e.g., a high Reynolds number in fluid mechanics or a large load increment in solid mechanics) that renders the system highly nonlinear and makes direct optimization via \cref{eq:newton} unstable. To circumvent this instability, the incremental continuation Newton method (IC-Newton) constructs a discrete parametric path $(\lambda_1, \lambda_2, \dots, \lambda_K)$ originating from a trivial or easily solvable state $\lambda_1$, and terminating at the target parameter $\lambda_K = \lambda^*$. The IC-Newton method then sequentially solves a sequence of problems $r(u; \lambda_{k+1}) = 0$ for $k = 1, \dots, K-1$, where the initial guess for the $(k+1)$-th step is taken as the converged solution of the $k$-th step, $u^*(\lambda_k)$. Provided that the partition of the parametric path is sufficiently fine, the initial guess for each subproblem remains within its respective 
local basin of attraction, ensuring robust convergence to the final target solution $u^*(\lambda^*)$~\cite{allgower2003introduction}.

\subsection{Neural operators}
In contrast to standard deep learning architectures that map between discrete, finite-dimensional Euclidean vectors, neural operators learn mappings between infinite-dimensional function spaces. Let $\mathcal{A}$ and $\mathcal{U}$ be bounded Banach spaces of functions defined over a compact spatial domain $\Omega \subset \mathbb{R}^d$. A neural operator $\mathcal{G}^\theta: \mathcal{A} \to \mathcal{U}$, parameterized by a vector of weights and biases $\theta \in \Theta$, aims to approximate an underlying nonlinear solution operator $\mathcal{G}$ that maps a parametric PDE input (e.g., varying boundary conditions, source terms, or spatial material profiles) to its corresponding solution field. 

The mathematical foundation of this paradigm rests upon the universal approximation theorems for operators~\cite{chen1995universal}, which guarantee that a neural operator with sufficiently wide or deep architectures can approximate any continuous operator on a compact set to arbitrary precision $\epsilon > 0$, such that $\sup_{a \in \mathcal{A}} \|\mathcal{G}(a) - \mathcal{G}^\theta(a)\|_{\mathcal{U}} < \epsilon$. Thus, given an input instance $a(\mathbf{x}) \in \mathcal{A}$, the neural operator infers the continuous solution field via a forward pass:
\[
    u(\mathbf{x}) \approx u^\theta(\mathbf{x}) = \left(\mathcal{G}^\theta a\right)(\mathbf{x}), \quad \mathbf{x} \in \Omega.
\]
Once optimized offline, evaluating $\mathcal{G}^\theta$ bypasses the traditional, computationally expensive assembly and inversion of high-dimensional algebraic matrices, yielding near-instantaneous functional inferences.

While various operator architectures have been proposed in literature~\cite{lu2021learning, li2021fourier, wu2024transolver}, this work primarily utilizes the Transolver~\cite{wu2024transolver} and the Fixed-Point Neural Operator (FPNO)~\cite{lee2025neural} frameworks.

\subsection{Neural operator-based hybrid Newton methods}
Hybrid computational frameworks seek to exploit the complementary advantages of data-driven surrogates and classical numerical analysis: the rapid, generalized approximation capabilities of neural operators and the exact, quadratically convergent refinement of classical Newton solvers. To effectively interface these distinct strategies, two prominent paradigms have emerged in the scientific machine learning literature.

\subsubsection*{Warm-started Newton method (WS-Newton)}
This strategy employs the neural operator strictly as an initial state generator for the classical Newton loop. Specifically, the neural operator prediction $u^\theta = \mathcal{G}^\theta(a)$ is assigned as the initial guess $u_0$ in \cref{eq:newton}, contrasting with the standard, uninformed choice of $u_0 = 0$. If $u^\theta$ successfully falls within the local basin of attraction of the target problem, the downstream classical Newton iteration immediately activates its quadratic convergence property, bypassing expensive incremental parametric continuation schedules. This convergence behavior is formalized below:

\begin{proposition}\label{prop:ws_newton}
Let $\mathcal{A}$ be the parameter space and $K \subset \mathcal{A}$ be a compact subset. 
Let $\mathcal{G}: K \to \mathcal{U}$ denote the exact, continuous solution operator mapping an input instance $a \in K$ to the true \textbf{discrete} solution $u$.
Assume that for each $a \in K$, the residual Jacobian $J(u; a)$ satisfies the following conditions within an open neighborhood $\mathcal{B}(\mathcal{G}(a), \rho_a) \subset \mathcal{U}$:
\begin{enumerate}
    \item The Jacobian operator is boundedly invertible at the exact solution:
    \[
        \|J(\mathcal{G}(a); a)^{-1}\| \le \alpha_a.
    \]
    \item The Jacobian operator is Lipschitz continuous with respect to the state variable:
    \[
        \|J(u; a) - J(v; a)\| \le \beta_a \|u - v\|, \quad \forall u, v \in \mathcal{B}(\mathcal{G}(a), \rho_a).
    \]
\end{enumerate}
Furthermore, assume that $\alpha_a$, $\beta_a$, and $\rho_a$ are strictly positive 
and depend continuously on $a \in K$. Then, for any parameter input $a \in K$, the WS-Newton 
method initialized by a sufficiently accurate neural operator exhibits local quadratic 
convergence to $\mathcal{G}(a)$.
\end{proposition}

\begin{proof}
For each parameter $a \in K$, the standard Newton--Kantorovich theorem asserts that Newton's method converges quadratically to the unique local solution $\mathcal{G}(a)$ provided the initial guess $u_0$ lies within the basin of attraction $\mathcal{B}(\mathcal{G}(a), \delta_a)$, where the radius of convergence is given by:
\[
    \delta_a = \min \left\{ \rho_a, \frac{1}{2\alpha_a \beta_a} \right\} > 0.
\]
Consider the mapping $a \mapsto \delta_a$. Because $\alpha_a$, $\beta_a$, and $\rho_a$ are strictly positive and vary continuously with $a$ over the compact domain $K$, the radius function $\delta_a$ is continuous and bounded away from zero. By the Extreme Value Theorem, there exists a uniform lower bound:
\[
    \varepsilon := \min_{a \in K} \delta_a > 0.
\]
By the universal approximation capabilities of neural operators over compact domains~\cite{chen1995universal}, for any given $\epsilon > 0$, there exists a neural operator $\mathcal{G}^\theta$ such that the uniform approximation error satisfies:
\[
    \sup_{a \in K} \|\mathcal{G}^\theta(a) - \mathcal{G}(a)\|_{C^0} < \varepsilon.
\]
Consequently, for all $a \in K$, the initialization provided by the neural operator satisfies $\|\mathcal{G}^\theta(a) - \mathcal{G}(a)\| < \delta_a$. Thus, the warm-started initial state falls strictly within the local quadratic convergence zone of the classical Newton solver.
\end{proof}

\subsubsection*{Neural operator preconditioned Newton method (NP-Newton)}
Pioneered by Lee et al.~\cite{lee2025neural}, this strategy treats the neural operator as a dynamic, continuous nonlinear preconditioner rather than a static initial guess generator. The formal NP-Newton updates are expressed as:
\begin{align}
    v_k     &= \mathcal{G}_{\text{FPNO}}(u_k, r(u_k); \theta), \label{eq:np_newton_a} \\
    u_{k+1} &= v_k - \eta_k J(v_k)^{-1} r(v_k), \label{eq:np_newton_b}
\end{align}
where $v_k$ represents the intermediate state preconditioned dynamically by the fixed-point neural operator ($\mathcal{G}_{\text{FPNO}}$), and $u_{k+1}$ is the subsequent classical Newton correction computed utilizing the preconditioned variable.

Both hybrid formulations have demonstrated impressive empirical acceleration across several examples~\cite{taghikhani2025neural, lee2025neural}. However, their numerical robustness remains largely unmapped. A comprehensive understanding of the failure modes inherent to these emerging techniques is a vital prerequisite for engineering reliable hybrid numerical solvers. In the subsequent section, we evaluate the limits of both the WS-Newton and NP-Newton methods by exposing them to highly challenging, strongly nonlinear physical regimes.

\section{Failure modes in hybridization}\label{sec:examples}
In this section, we systematically dissect the primary numerical failure modes encountered when hybridizing neural operators with classical Newton's method. We begin by analyzing a low-dimensional, highly nonlinear algebraic system designed to isolate specific algorithmic vulnerabilities.

\subsection{A pedagogical toy example}\label{sec:pedagogical_example}
To isolate and systematically investigate the numerical pathologies and structural failure modes associated with hybrid Newton methods, we first examine a highly nonlinear, two-dimensional algebraic system of equations. We define the discrete system residual $r(\mathbf{x}) = 0$ as:
\[
    r(\mathbf{x}) = r(x, y) =
    \begin{bmatrix}
        (x - y^3 + 1)^5 - y^5 \\
        x + 2y - 3
    \end{bmatrix},
\]
which possesses an exact analytical solution at $\mathbf{x}^* = (1, 1)^\top$. Initializing the classical Newton solver from an unguided guess of $\mathbf{x}_0 = (2, 2)^\top$ presents a severe numerical challenge; owing to the sharp cubic and quintic nonlinearities embedded within the first residual component, the standard Newton method requires 13 iterations to satisfy standard convergence criteria. 

The left panel of \cref{fig:pedagogy} maps the corresponding residual norm landscape overlaid with this discrete Newton trajectory. To establish a quantitative diagnostic for this severe scaling imbalance, we evaluate the Euclidean $L^2$-norm condition number, $\kappa(J)$, of the system Jacobian matrix at each discrete step $k$. At the initial state $\mathbf{x}_0$, the condition number is excessively high, registering at $\kappa \approx 3.24 \times 10^4$. As the iterative trajectory approaches the root, the spectrum stabilizes, compressing $\kappa$ to a well-behaved magnitude of $O(10^1)$. Correlating these algebraic metrics with the geometric features shown in \cref{fig:pedagogy}, a large condition number manifests as a highly stretched, anisotropic, and narrow valley, whereas a low condition number indicates an isotropic, balanced residual landscape. Importantly, we emphasize that this spectral interpretation remains strictly local to the current state iterate.

\begin{figure}[htbp]
    \centering
    \includegraphics[width=1\linewidth]{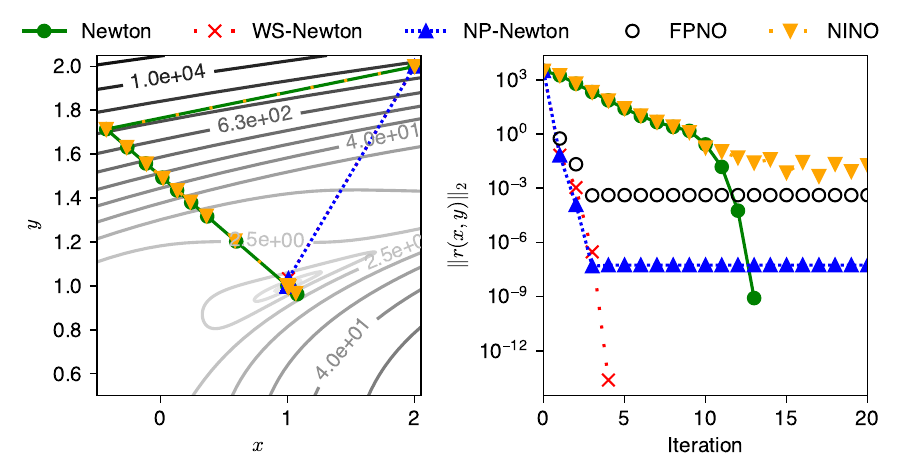}
    \caption{Numerical behavior and optimization trajectories for the pedagogical toy problem. Left: Two-dimensional residual norm landscape $\|r(x,y)\|_2$ overlaid with the iterative paths of the classical Newton, WS-Newton, and NP-Newton solvers. The classical Newton method exhibits initial stagnation due to severe nonlinearity along the $x$-coordinate direction, whereas the hybrid methods leverage operator predictions to place the initial state near the exact root $\mathbf{x}^*$. Right: Residual convergence histories across iterations. While the well-initialized WS-Newton framework triggers immediate asymptotic quadratic convergence, an undertrained neural operator causes the dynamic preconditioning loop in the NP-Newton method to exhibit severe numerical oscillations.}
    \label{fig:pedagogy}
\end{figure}

\subsubsection*{Dataset and neural operator training}
To train the operator, we collect the discrete state trajectories $(\mathbf{x}_k)_{k=0}^K$ generated during the successful convergence of the classical Newton method from $\mathbf{x}_0 = (2, 2)^\top$ to the exact root $\mathbf{x}_K = \mathbf{x}^* = (1, 1)^\top$. To fulfill the input requirements of the FPNO framework utilized in the NP-Newton method, we simultaneously record the corresponding algebraic residual trajectory $(r_k)_{k=0}^K$, where $r_k = r(\mathbf{x}_k)$. Both the backbone and scaling networks of the FPNO architecture are parameterized as Multi-Layer Perceptrons (MLPs) featuring a uniform hidden width of 10 channels, a depth of 3 layers, and the hyperbolic tangent activation function. The neural operator is trained with Adam optimizer~\cite{kingma2014adam} by minimizing the mean squared error (MSE) loss function:
\[
    \mathcal{L}(\theta) = \frac{1}{K}\sum_{k=0}^K \left\| \mathbf{x}^* - \mathcal{G}_{\text{FPNO}}(\mathbf{x}_k, r_k; \theta) \right\|_2^2.
\]
Owing to the low-dimensional nature of the system, this optimization phase requires negligible computational effort: within seconds on an Apple M3 Central Processing Unit (CPU) with JAX~\cite{jax2018github}.

\subsubsection*{Numerical result}
We first compare the hybrid methods with a fully data-driven Newton-informed neural operator (NINO)~\cite{hao2024newton} that approximates the Newton direction \(-J({\bf x})^{-1}r({\bf x})\) by a neural operator. As illustrated in the right panel of \Cref{fig:pedagogy}, the NINO fails to reduce the residual norm below \(10^{-3}\) and stagnates. By contrast, when the FPNO is fully optimized to a tight training tolerance (e.g., $\mathcal{L}(\theta) < 10^{-4}$), both the WS-Newton and NP-Newton methods successfully locate the exact analytical solution within a single iteration. As illustrated in the left panel of \cref{fig:pedagogy}, the optimized neural operator maps the unguided initial guess $\mathbf{x}_0$ directly into the immediate neighborhood of the root $\mathbf{x}^*$. By bypassing the ill-conditioned subregions governed by severe, unbalanced quintic nonlinearities, the hybrid Newton methods trigger immediate convergence.

An instructive numerical failure mode emerges, however, if the neural network training is intentionally terminated early, thereby introducing minor prediction errors into the operator's inference space. To simulate this scenario, optimization is halted when the empirical loss reaches a threshold of $\mathcal{L}(\theta) \approx 10^{-3}$. Under this regime of imperfect surrogate representation, the behavior of the hybrid Newton methods becomes different. As demonstrated by the residual histories in the right panel of \cref{fig:pedagogy}, the standard Newton and WS-Newton methods exhibit standard asymptotic quadratic convergence. Conversely, the alternating NP-Newton loop stagnates, near $\|r(\mathbf{x}_k)\|_2 \approx 10^{-8}$.

A detailed algorithmic inspection of the NP-Newton updates defined in \Cref{eq:np_newton_a,eq:np_newton_b}\footnote{For the two-dimensional pedagogical system, the state variable $u$ is substituted by $\mathbf{x}$.} clarifies the mechanics of this stagnation. Although the Newton step~\cref{eq:np_newton_b} effectively reduces the residual norm from $O(10^{-4})$ at the preconditioned state $v_k$ down to $O(10^{-8})$ at the next state iterate $\mathbf{x}_{k+1}$, the subsequent unregularized evaluation~\cref{eq:np_newton_a} of the imperfect operator, $v_{k+1} = \mathcal{G}_{\text{FPNO}}(\mathbf{x}_{k+1}, r_{k+1}; \theta)$, systematically corrupts the trajectory. This step reinjects a localized prediction error that inflates the intermediate residual norm back above $10^{-4}$. This oscillatory cycle prevents the hybrid solver from satisfying strict convergence tolerances, underscoring the critical necessity of a rich dataset that can cover Newton trajectories in various scenarios.

While this low-dimensional example provides valuable intuition regarding the oscillatory instabilities inherent to alternating preconditioned updates, the acceleration achieved by the WS-Newton method remains trivial because the underlying baseline system converges in fewer than 20 iterations. In the subsequent section, we evaluate these hybrid paradigms against a significantly more challenging benchmark wherein the classical Newton solver completely fails to converge.

\subsection{Lid-driven Cavity Flow}\label{sec:ldc}
To evaluate the performance of the proposed framework on a fluid dynamics benchmark, we consider the steady-state incompressible Navier--Stokes equations modeled via the stream function--vorticity formulation. The physical domain is defined as a unit square $\Omega = (0, 1)^2$ with a closed boundary $\Gamma = \partial\Omega$. The boundary conditions for the velocity field $u = (u_x, u_y)^\top$ prescribe a regularized top-driven lid, expressed as:
\[
    u_x = 
    \begin{cases} 
    1, & \text{on } \Gamma_{\text{top}} = \{(x, y) \in \bar{\Omega} : y = 1\}, \\
    0, & \text{on } \Gamma \setminus \Gamma_{\text{top}},
    \end{cases}
    \quad \text{and} \quad u_y = 0 \quad \text{on } \Gamma.
\]
To resolve the discontinuities at the two corners, we use \(1 - \tfrac{\cosh(r(x - 0.5))}{\cosh(0.5r)}\) with \(r = 50\) for smoothing. The governing coupled steady-state partial differential equations are formulated as:
\begin{align}
(u \cdot \nabla)\omega &= \frac{1}{\mathrm{Re}}\nabla^2\omega, \label{eq:ns_vorticity} \\
-\nabla^2 \psi &= \omega, \label{eq:ns_poisson} \\
u &= \left(\frac{\partial \psi}{\partial y}, -\frac{\partial \psi}{\partial x}\right)^\top, \label{eq:ns_velocity}
\end{align}
where $\mathrm{Re}$ denotes the dimensionless Reynolds number and the scalar vorticity is related to the velocity field via $\omega = \nabla \times u = \frac{\partial u_y}{\partial x} -\frac{\partial u_x}{\partial y}$. This system constitutes the classic lid-driven cavity flow benchmark problem~\cite{ghia1982high}.

As the Reynolds number increases, the diffusive term in \cref{eq:ns_vorticity} becomes progressively dominated by the nonlinear convective transport term. From a spectral perspective, this convective dominance makes the vorticity field $\omega$ decay slowly, implying the generation of high-frequency spatial modes that complicate the overall solution profile. Physically, this manifests as the formation of thin, highly sharp boundary layers along the solid walls and the emergence of eddies at the corners. To accurately resolve these high-gradient boundary layers, the spatial domain is discretized using a high-order spectral element method~\cite{karniadakis2005spectral} with the scheme described in \cite{liu2001simple}. We employ an $8^{\text{th}}$-order nodal expansion over a structurally graded algebraic mesh comprising $51 \times 51$ nodes clustered near the boundaries. \Cref{fig:ldc-mesh} illustrates the non-uniform discretization alongside the steady-state stream function and vorticity fields computed at $\mathrm{Re} = 7500$.

\begin{figure}[ht]
    \centering
    \includegraphics[width=1\linewidth]{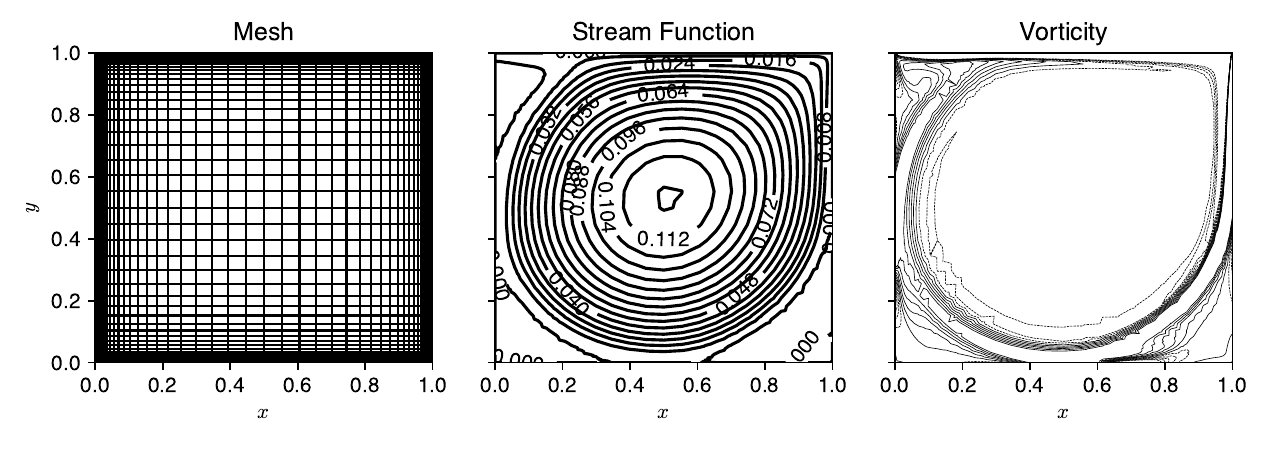}
    \caption{Numerical discretization and steady-state solution profiles for the lid-driven 
    cavity flow at $\mathrm{Re}=7500$. Left: Structurally graded mesh topology designed to resolve boundary layers. Center: Two-dimensional contour plot of the computed stream function $\psi$. Right: Two-dimensional contour plot of the resolved vorticity field $\omega$. 
    The underlying nonlinear system was converged to $\|r\|_2 \approx 1.94 \times 10^{-14}$.}
    \label{fig:ldc-mesh}
\end{figure}

\subsubsection*{Dataset and neural operator training}
For dataset generation, we sample $M = 200$ dimensionless Reynolds numbers from a uniform distribution, $\mathrm{Re} \sim \mathcal{U}(10^3, 10^4)$. To overcome the severe convective nonlinearities encountered at high Reynolds numbers, the IC-Newton method (\cref{sec:ic-newton}) is employed using $20$ uniform loading steps. Specifically, the numerical solver sequentially steps the system from an initial state of $\mathrm{Re}/20$ up to the target value $\mathrm{Re}$, recording the intermediate stream functions $\psi_k$ and algebraic residual vectors $r_k = r(\psi_k, \omega_k)$ at each step.

A FPNO ($\mathcal{G}_{\text{FPNO}}$), configured with a Transolver++~\cite{luo2025transolver++} backbone and scaling network architecture, is trained to predict the steady-state stream function field $\psi$. The neural operator is optimized via a component-wise relative $L^2$ loss function formulated as:
\[
    \mathcal{L}(\theta) = \frac{1}{M}\sum_{m=1}^M \left( \frac{1}{K_m}\sum_{k=0}^{K_m} \left\| \frac{\psi^{*, (m)} - \mathcal{G}_{\text{FPNO}}(\psi_k^{(m)}, r_k^{(m)}; \theta)}{|\psi^{*, (m)}| + \varepsilon} \right\|_2^2 \right),
\]
where the superscript $(m)$ denotes the $m$-th randomly sampled Reynolds number instance, $\psi^{*, (m)}$ is the corresponding ground-truth high-fidelity stream function solution, and $\varepsilon = 10^{-2}$ is a small regularization parameter preventing division by zero. The associated velocity field $u$ and vorticity field $\omega$ are subsequently reconstructed from the predicted stream function $\psi$ via exact numerical differentiation utilizing \cref{eq:ns_velocity}.

The network optimization is executed using the AdamW~\cite{kingma2014adam, loshchilov2018decoupled} optimizer with a batch size of 100, a learning rate of $10^{-4}$, and a weight decay of $5 \times 10^{-4}$. Training is terminated after achieving a mean relative $L^2$ error of $5 \times 10^{-3}$ on the unseen test dataset. The total training phase required approximately 4.19 hours of computational time on a single NVIDIA H100 Graphics Processing Unit (GPU).

\subsubsection*{Numerical result at a high Reynolds number}
The performance of the hybrid schemes is evaluated at a target test Reynolds number of $\mathrm{Re} = 7500$. Although the FPNO is trained across the broader parametric range $\mathrm{Re} \sim \mathcal{U}(10^3, 10^4)$, evaluating steady-state solvers beyond $\mathrm{Re} \ge 8000$ introduces physical anomalies; at these regimes, the flow undergoes a transition to turbulence, meaning a steady-state solution lacks physical realizability~\cite{shankar2000fluid, erturk2009discussions}. The left panel of \Cref{fig:ldc-convergence} details the relative residual convergence profiles for both the WS-Newton and NP-Newton formulations.

\begin{figure}[ht]
    \centering
    \includegraphics[width=1\linewidth]{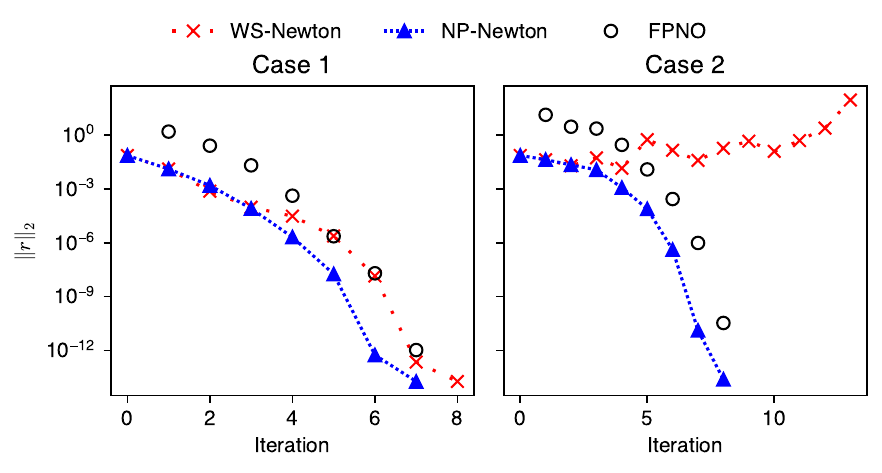}
    \caption{Residual norm convergence profiles for hybrid and classical numerical solvers 
    under varying training limits and flow regimes. Left (Case 1): strict error tolerance 
    $\mathcal{L}_{\text{tol}} = 5 \times 10^{-3}$ at $\mathrm{Re}=7500$. Right (Case 2): relaxed 
    error tolerance $\mathcal{L}_{\text{tol}} = 5 \times 10^{-2}$ at $\mathrm{Re}=3000$.}
    \label{fig:ldc-convergence}
\end{figure}

\Cref{tab:cavity_convergence} provides a detailed quantitative comparison of the total cumulative iteration counts and the corresponding wall-clock times required to satisfy the convergence criteria. When benchmarked against the IC-Newton method requiring 65 cumulative iterations across 15 loading steps, both neural operator-based hybrid frameworks achieve an approximate $10\times$ computational speedup. Notably, despite requiring one additional iteration, the WS-Newton method yields the lowest total wall-clock time ($83.8$~seconds). This efficiency stems from bypassing the repeating intermediate inference evaluations of the neural operator that are embedded within the active iterative loop of the NP-Newton scheme.

\begin{table}[t]
\centering
\caption{Computational performance comparison for the lid-driven cavity flow at $\mathrm{Re} = 7500$. A direct linear solver is used on a single core.}
\label{tab:cavity_convergence}
\resizebox{\columnwidth}{!}{
\begin{tabular}{lcccc}
\hline
\textbf{Method} & $\|r\|_2$ & \textbf{Iterations} & \textbf{Wall Time} (s) & \textbf{Speedup} \\ \hline
IC-Newton (15 steps) & $1.11 \times 10^{-12}$ & 65 & 920.0 & $1.00$ (baseline) \\
NP-Newton            & $1.86 \times 10^{-14}$ & 7  & 97.9  & $9.40\times$      \\
WS-Newton            & $1.85 \times 10^{-14}$ & 8  & 83.8  & $11.0\times$     \\ \hline
\end{tabular}
}
\end{table}

\subsubsection*{Robustness under imperfect surrogates}
To systematically investigate the limits of algorithmic robustness, we relax the neural operator's training termination criterion from a strict learning tolerance ($\mathcal{L}_{\text{tol}} = 5 \times 10^{-3}$) to a relaxed threshold ($\mathcal{L}_{\text{tol}} = 5 \times 10^{-2}$). Concurrently, the evaluation suite is expanded across a broader range of flows: $\mathrm{Re} \in \{1500, 3000, 4500, 6000, 7500\}$. \Cref{tab:cavity-robust} summarizes the resulting convergence characteristics.\footnote{For a severely relaxed learning tolerance of $\mathcal{L}_{\text{tol}} = 1 \times 10^{-1}$, both the WS-Newton and NP-Newton formulations diverged immediately across all tested regimes.}

Importantly, the performance under an undertrained neural operator contrasts with the low-dimensional findings in \cref{sec:pedagogical_example}. While the static WS-Newton framework fails to converge for all regimes where $\mathrm{Re} \ge 3000$, the dynamically preconditioned NP-Newton loop successfully resolves the flow fields up to $\mathrm{Re} = 4500$. This phenomenon indicates that the NP-Newton formulation exhibits greater algorithmic resilience when trained with sufficient data that includes various Newton trajectories. This robustness can be attributed to FPNO's continuous structural, state-dependent corrections at every iterative step rather than relying on a single, static initialization. 

The right panel of \cref{fig:ldc-convergence} tracks the step-by-step residual histories for both hybrid approaches at $\mathrm{Re} = 3000$. Circles indicate neural operator outputs. Cross and triangle markers denote Newton updates. While the WS-Newton path experiences divergence starting at the fifth iteration, the FPNO correction in the NP-Newton loop prevents the divergence, successfully driving the solution below the predefined convergence tolerance.

\begin{table}[t]
\centering
\caption{Algorithmic robustness for the lid-driven cavity flow under varying training tolerances ($\mathcal{L}_{\text{tol}}$) and Reynolds numbers ($\mathrm{Re}$). Data entries correspond to iteration counts formatted as (WS-Newton / NP-Newton) required to satisfy a strict relative residual convergence threshold of $\|r\|_2 / \|r_0\|_2 < 10^{-9}$. The label ``--'' designates failure to converge.}
\label{tab:cavity-robust}
\begin{tabular}{l|ccccc}
\toprule
\multirow{2}{*}{$\mathcal{L}_{\text{tol}}$} & \multicolumn{5}{c}{$\mathrm{Re}$}\\
& 1500 & 3000 & 4500 & 6000 & 7500 \\\midrule

$5 \times 10^{-2}$ & 7 / 7 & -- / 8 & -- / 8 & -- / --& -- / -- \\
$1 \times 10^{-2}$ & 5 / 6 & 5 / 7   & 6 / 7   & 6 / 8     & 6 / 8     \\
$5 \times 10^{-3}$ & 5 / 5 & 5 / 5   & 5 / 6   & 6 / 6     & 8 / 7     \\ \bottomrule
\end{tabular}
\end{table}

\Cref{tab:cavity-robust} presents the overall result, confirming that while both hybrid approaches perform well when paired with a highly accurate surrogate ($\mathcal{L}_{\text{tol}} = 5 \times 10^{-3}$), their behaviors diverge as surrogate fidelity degrades. Notably, the WS-Newton method typically required fewer iterations to converge, if it started to converge. The only exception was \(\mathcal{L}_\text{tol} = 5\times10^{-3}\) and \(\mathrm{Re} = 7500\) case.

In the subsequent section, we move beyond these canonical scenarios to expose a far more severe class of numerical failure modes: scenarios where hybridization breaks down even when paired with a fully optimized, highly accurate neural operator.

\subsection{2D hyperelasticity}\label{sec:2dhel}
In this section, we revisit the 2D hyperelasticity problem considered in \cite[Section 3.2]{lee2025neural}. Let $\Omega \subset [0, 1]^2$ denote the reference domain in a unit square. \Cref{fig:hel_mesh} visualizes the domain. The total potential energy functional $\Pi(u)$ in terms of the displacement field $u$ is given by:
\begin{equation}\label{eq:hel}
    \Pi(u) = \int_\Omega \left[ \frac{\mu}{2}(I_c - 3) - \mu \ln(\det \mathbf{F}) + \frac{\lambda_{\text{Lam\'e}}}{2}(\det \mathbf{F} - 1)^2 \right] d\mathbf{x},
\end{equation}
where $\mathbf{F} = \mathbf{I} + \nabla u$ represents the deformation gradient tensor, $\mathbf{C} = \mathbf{F}^\top\mathbf{F}$ is the right Cauchy--Green deformation tensor, and $I_c = \mathrm{tr}(\mathbf{C})$ is its first invariant. The material parameters are defined via Lam\'e constants calculated from Young's modulus $E = 1.0$ and Poisson's ratio $\nu = 0.49$:
\begin{equation}\label{eq:lame}
    \mu = \frac{E}{2(1+\nu)}, \qquad \lambda_{\text{Lam\'e}} = \frac{E\nu}{(1+\nu)(1-2\nu)}.
\end{equation}
Note that as \(\nu\) approaches to \(0.5\), \(\lambda_{\text{Lam\'e}}\) approaches to \(\infty\), enforcing the incompressibility condition \(\det \mathbf{F} = 1\) strictly. Thus, $\nu = 0.49$ places the system in a nearly incompressible limit, prone to extreme numerical instability.

\begin{figure}[t]
    \centering
    \includegraphics[width=0.75\linewidth]{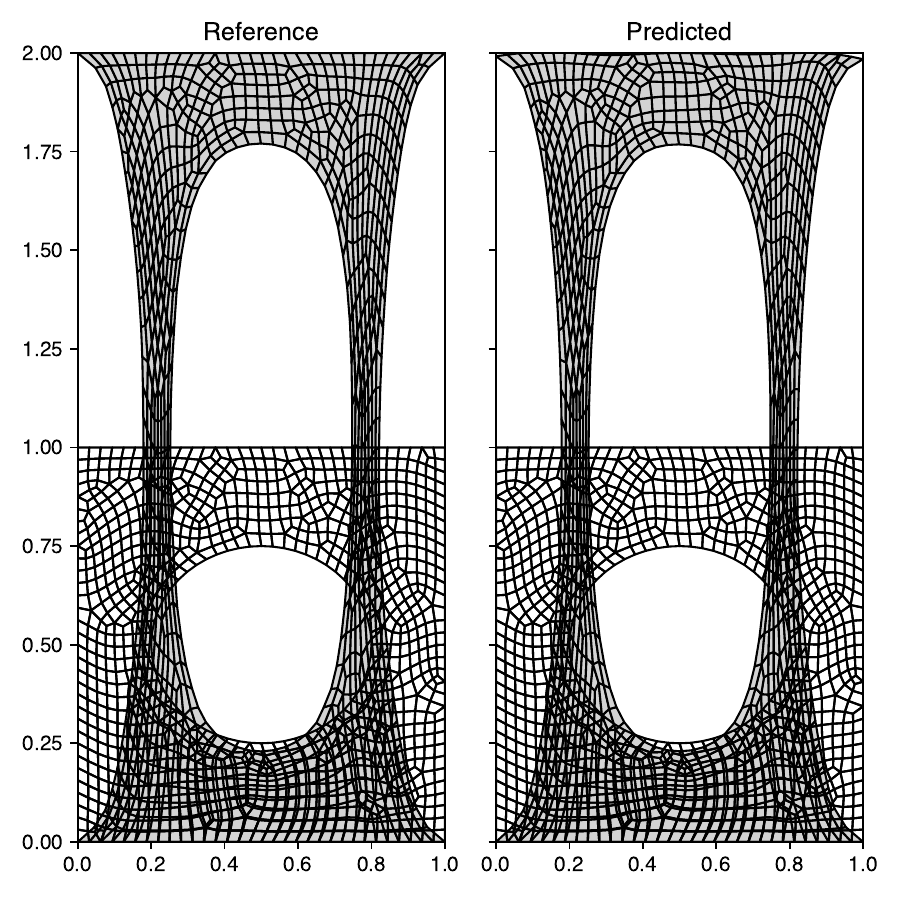}
    \caption{Overlay of the original domain and deformed domain for \(\tau = 1\). Left: reference deformation found by the reference solver. Right: deformation predicted by a neural operator. Both reference deformation and predicted deformation look similar, but there are slight gaps in the top corners.}
    \label{fig:hel_mesh}
\end{figure}

\Cref{eq:hel} is discretized using standard $P^1$ Lagrange finite elements over a structured tetrahedral mesh $\mathcal{T}_h$, defining the discrete space:
\[
    \mathcal{V}_h = \{u \in [H^1(\Omega)]^2 : u|_T \in [P^1(T)]^2 \quad \forall T \in \mathcal{T}_h\}.
\]
The boundary conditions are split into a clamped bottom domain $\Gamma_1 = \{(x,y) \in \Omega : y = 0\}$ and a displaced top surface $\Gamma_2 = \{(x,y) \in \Omega : y = 1\}$ such that:
\[
    u = (0, 0)^\top \quad \text{on } \Gamma_1, \qquad u = (0, \tau)^\top \quad \text{on } \Gamma_2.
\]

We seek to find the deformed state by
\[
    u^* = \arg \min_{\bf u} \Pi(u).
\]
We solve the minimization problem with the first-order optimality condition, i.e., finding the solution to
\[
    r(u) := \nabla_u \Pi(u) = 0.
\]
We evaluate the performance of our framework under a computationally demanding many-query setup, where the top boundary displacement $\tau$ is varied continuously within the interval $[0, 2]$. Large displacement fields ($\tau \ge 1.0$) introduce severe geometric and material nonlinearities, making standard iterative methods highly sensitive to initial conditions.

\subsubsection*{Numerical result with FPNO}
Following the framework established by Lee et al.~\cite{lee2025neural}, a training dataset was originally constructed by sampling initial state guesses from Gaussian random fields and parametric boundary conditions from a uniform distribution, $\tau \sim \mathcal{U}(0, 2)$. Optimization trajectories were subsequently collected from a baseline Newton solver whose step size $\eta_k$ is dynamically found at each iteration by a trust-region method (Newton-TR). The FPNO is trained to map these current iterates and their corresponding residuals directly to their final converged solutions, aiming to bypass ill-conditioned regions dominated by highly unbalanced nonlinearities. When evaluated at a specific loading parameter of $\tau = 1$, the dynamically preconditioned NP-Newton loop reportedly yields a $6.98\times$ computational acceleration over the standalone Newton-TR baseline~\cite[Table 3]{lee2025neural}, utilizing the MUMPS sparse direct LU solver framework.

While this structural acceleration is conceptually compelling, it remains critical to identify precisely whether the continuous online corrections or merely the initial warm start drives the performance within the NP-Newton loop. To isolate these mechanisms, we conduct an algorithmic ablation study wherein the FPNO preconditioning is systematically deactivated after the initial iteration ($k > 0$), effectively reducing the framework to a static WS-Newton method. \Cref{tab:sqh_robust} details the results of this robustness test evaluated across an escalated loading regime: $\tau \in \{1, 1.25, 1.5, 1.6, 1.75, 1.8, 2\}$. 

Our ablation study reveals two critical numerical phenomena. First, in the rare instances where the WS-Newton method successfully converges, it requires fewer total iterations than the fully preconditioned NP-Newton scheme, bypassing the minor overhead of iterative surrogate evaluations. Second, and more importantly, both hybrid approaches exhibit a profound lack of robustness compared to the Newton-TR baseline as $\tau$ increases. Despite being explicitly trained on trajectories spanning $\tau \sim \mathcal{U}(0, 2)$, both the WS-Newton and NP-Newton formulations suffer from divergence across the majority of the higher loading parameters. This vulnerability highlights a fundamental generalization collapse when the operator encounters highly strained states.

\begin{table}[t]
\centering
\caption{Robustness and numerical result for 2D hyperelasticity with FPNO. Values denote iterations required to achieve a relative residual norm below $10^{-9}$. The symbol ``--'' indicates solver divergence or stagnation.}
\label{tab:sqh_robust}
\begin{tabular}{l|ccccccc}
\toprule
\multirow{2}{*}{\textbf{Method}} & \multicolumn{7}{c}{$\tau$}\\
& 1.00 & 1.25 & 1.50 & 1.60 & 1.75 & 1.80 & 2.00 \\\midrule
Newton & 107 & 133 & 159 & 169 & 185 & 190 & 211 \\
WS-Newton & 5   & -- & 6   & -- & -- & -- & -- \\
NP-Newton & 6   & 7   & 9   & -- & 10  & -- & -- \\ \bottomrule
\end{tabular}
\end{table}

\subsubsection*{Scalability limit of the NP-Newton method}
The total number of DOFs for this low-dimensional validation case is only 2,059. Within this small-scale regime, the baseline Newton-TR solver requires a few seconds on a single CPU, even for the most severe loading state of $\tau = 2$. When factoring in the immense offline computational cost required for trajectory data collection and subsequent operator training, neural operator-based hybrid solvers become practically uncompetitive unless they can be scaled up to multi-million DOF engineering systems. 

However, scaling the NP-Newton or FPNO frameworks to large-scale 3D regimes exposes a severe data-storage and scalability bottleneck. As shown in \Cref{tab:sqh_robust}, a single unguided Newton trajectory requires approximately $100$ discrete steps on average to reach convergence. Sampling a modest group of $200$ parameter variations requires storing approximately $20,000$ distinct state-residual-solution triplets $(u_k, r_k, u^*)$ to assemble the training dataset. For large-scale 3D models where a single state vector demands much bigger memory, storing full iterative trajectories becomes entirely prohibitive. Consequently, the NP-Newton method that depends fundamentally on full trajectory histories suffers from severe scalability limits. To achieve true computational scalability, we must pivot exclusively to a warm-started framework driven by a parameter-to-state operator architecture trained on a direct static mapping $\tau \mapsto u^*$, which requires only a single converged snapshot per parameter sample.

Furthermore, moving to large-scale simulations imposes a strict constraint: iterative linear solvers become favorable due to the prohibitive memory footprint of sparse direct solvers. While finite element discretizations naturally preserve Jacobian sparsity, direct factorization scales poorly with domain size. Recognizing this imperative, we explore the employment of iterative linear solvers and study associated pathologies in the WS-Newton method.

\subsubsection{Towards a scalable WS-Newton method}\label{sec:spectral_analysis}
To establish a computationally scalable framework for the WS-Newton solver, we pivot from an iterative framework to a direct parameter-to-state mapping, approximating the static solution operator $\tau \mapsto u^*$. We utilize the Transolver architecture~\cite{wu2024transolver} rather than an FPNO, thereby completely bypassing the requirement for storing full, multi-step Newton trajectories. Specifically, we adopt the memory-efficient Transolver++ variant~\cite{luo2025transolver++}, wherein we substitute the stochastic Gumbel-Softmax trick with a standard, deterministic Softmax function to ensure permutation equivariance.

To construct the training dataset, we sample $M = 300$ parametric loading states from a uniform distribution, $\tau \sim \mathcal{U}(0,2)$, and record only the final, converged high-fidelity snapshots to assemble the dataset $\{(\tau^{(m)}, u^{*,(m)})\}_{m=1}^{300}$. By shifting from full trajectory sequences to static end-state snapshots, the data-storage footprint is successfully reduced by two orders of magnitude. The operator is subsequently pre-trained on this dataset to a relative $L^2$ learning tolerance of $\mathcal{L}_{\text{tol}} = 3 \times 10^{-3}$.

\begin{table}[t]
\centering
\caption{Robustness result and step counts for the WS-Newton method under pre-trained and fine-tuned regimes. Convergence is defined by achieving a relative residual tolerance of $\|r\|_2 / \|r_0\|_2 < 10^{-9}$ under trust-region method. The symbol ``--'' denotes solver divergence or stagnation.}
\label{tab:sqh_ft_robust}
\resizebox{\columnwidth}{!}{
\begin{tabular}{cc|ccccccc}
\toprule
\multirow{2}{*}{\bf neural operator}&\multirow{2}{*}{\bf linear solver}& \multicolumn{7}{c}{\(\tau\)}\\
& & 1.00 & 1.25 & 1.50 & 1.60 & 1.75 & 1.80 & 2.00 \\\midrule
\multirow{3}{*}{pretrained}&CG &-- &-- &-- &-- &-- &-- &--\\
&FGMRES &-- &8 &8 &-- &-- &-- &5\\
&LU & 8 &-- &-- &-- &-- &-- &14\\\midrule
\multirow{3}{*}{finetuned}&CG & 4 &5 &5 &5 &6 &6 &--\\
&FGMRES&4&4&5&5&6&6&6\\
&LU &5&5&6&6&6&6&7\\\bottomrule
\end{tabular}
}
\end{table}

The first three rows of Table~\ref{tab:sqh_ft_robust} outline the convergence performance of the pre-trained operator when used as an initializer. Because this neural operator approximates the direct map $\tau \mapsto u^*$, it is evaluated exclusively within the static WS-Newton framework. To rigorously evaluate scalability constraints, we combine the framework with both Krylov subspace iterative linear solvers -- specifically, the Conjugate Gradient (CG) method and Flexible GMRES (FGMRES) -- and a baseline direct LU factorization method, all backed by a trust-region method. 

The trained neural operator exhibits a poor robustness profile similar to the FPNO configurations, converging successfully in only a few cases. The pre-trained neural operator triggers divergence across all test cases when paired with the CG linear solver. To uncover the structural drivers of these convergence behaviors, we isolate and inspect the spectral properties of the initial system Jacobian matrix under three distinct initialization strategies: a default zero initial guess ($u_0 = \mathbf{0}$), a baseline state tracking the solution from the previous parameter step ($u_{\tau=0.9}$), and the raw inference field provided by the data-driven pre-trained operator. \Cref{fig:sqh_eigenvalues} plots the 20 smallest eigenvalues of the discrete system Jacobian evaluated at these states for $\tau = 1$. 

For the reference state $u_{\tau=0.9}$, the minimum eigenvalue $\lambda_{\min}$ remains strictly positive, satisfying classical local convex minimization requirements. Conversely, the Jacobian evaluated at the pre-trained operator's prediction yields 9 distinct negative eigenvalues. This indicates that, while the continuous variational formulation guarantees a symmetric and positive-definite (SPD) system Jacobian near the solution, the unregularized, raw neural network predictions inject localized errors. The right panel of \Cref{fig:hel_mesh} presents the raw neural operator prediction for \(\tau = 1\) -- such localized errors are difficult to capture by the eyes. These localized defects distort the discrete Jacobian, destroying its positive definiteness and causing the underlying CG algorithm to stall or diverge.

\begin{figure}[htbp]
    \centering
    \includegraphics[width=0.7\linewidth]{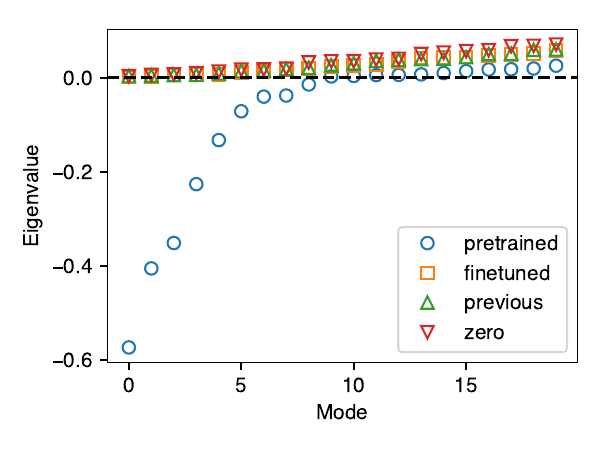}
    \caption{The 20 smallest eigenvalues of the initial system Jacobian matrices evaluated at $\tau = 1$ across different initialization configurations: the baseline zero guess, the previous step reference ($\tau = 0.9$), the unregularized pre-trained neural operator, and the physics-regularized fine-tuned operator.}
    \label{fig:sqh_eigenvalues}
\end{figure}

To connect this pathology to the underlying continuum mechanics, we evaluate the local volume change, $\det\mathbf{F}$. \Cref{fig:sqh_J} visualizes the spatial distribution and statistical spread of $\det\mathbf{F}$ at $\tau = 1$. As presented in the first and third panels, the $\det\mathbf{F}$ values of the pre-trained operator field deviate significantly from the reference solution. The histogram in the fourth panel highlights that the unregularized neural operator yields a broad dispersion of local volume changes. While the maximum values coincide near $1.20$, the minimum value is approximately $0.40$ for the pre-trained neural operator compared to a value of $0.80$ in the reference solution. 

We hypothesize that these severe localized fluctuations in unpenalized volume change can destroy the global spectrum of the Jacobian. To regularize this defect, we fine-tuned the pre-trained neural operator with a physics-informed loss targeting the minimization of the incompressibility constraint penalty:
\[
    \mathcal{L}_{\text{PI}} = \frac{\lambda_{\text{Lam\'{e}}}}{2}\int_\Omega \left( \det\mathbf{F}(\mathbf{x}) - 1 \right)^2 \, \mathrm{d}\mathbf{x}.
\]
In particular, the loss derivatives with respect to the state vector, \(\partial \mathcal{L}_{\text{PI}} / \partial u^\theta\), are computed via efficient finite-element routines executed on the CPU. These terms are subsequently passed to the GPU to complete backpropagation through the network parameters via automatic differentiation:
\[
    \frac{\partial \mathcal{L}}{\partial \theta} = \underbrace{\frac{\partial \mathcal{L}}{\partial u^\theta}^\top}_{\text{FEM (CPU)}} \,\, \underbrace{\frac{\partial u^\theta}{\partial \theta}}_{\text{AD (GPU)}}.
\]
The fine-tuning phase is executed over 200 epochs using 30 equispaced values of $\tau \in [0, 2]$, utilizing the AdamW optimizer with a learning rate of $5 \times 10^{-6}$. It requires negligible computational overhead, concluding in less than one minute on a single CPU.

\begin{figure}[ht]
    \centering
    \includegraphics[width=1\linewidth]{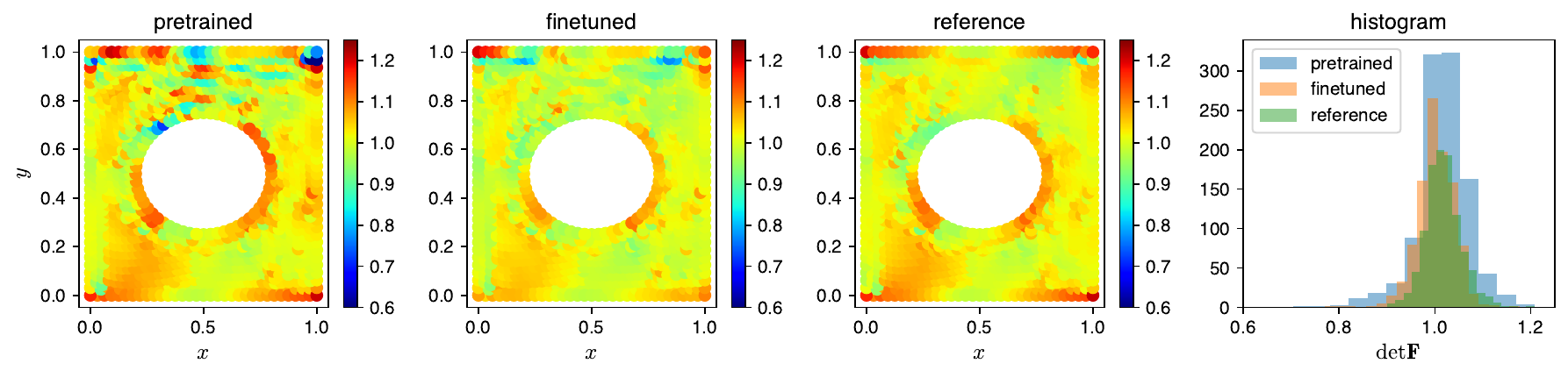}
    \caption{Spatial distributions and statistical frequencies of the deformation gradient determinant $\det\mathbf{F}$ at $\tau = 1$. Left to right: unregularized pre-trained operator prediction, regularized fine-tuned operator prediction, reference solution, and histograms across configurations.}
    \label{fig:sqh_J}
\end{figure}

The orange square markers in \cref{fig:sqh_eigenvalues} show the spectral recovery achieved following this physics-informed regularizing step. Every single eigenvalue shifts into the positive half-plane, establishing a strictly positive-definite system Jacobian. This spectral shift is explained by the \(\det\mathbf{F}\) values visualized in the second panel of \cref{fig:sqh_J}; the field from the fine-tuned neural operator is similar to that of the reference solution.

As a direct consequence of this restored spectrum, the WS-Newton framework paired with the CG iterative linear solver achieves complete convergence across all loading parameters, with the sole exception of the extreme $\tau = 2$ case where the system requires a flexible CG variant. Importantly, as outlined in the lower blocks of \Cref{tab:sqh_ft_robust}, this data-free regularized WS-Newton approach significantly outperforms the trajectory-trained NP-Newton paradigm, achieving both superior iteration reduction and an eliminated storage footprint.

\subsection{3D hyperelasticity}\label{sec:3dhel}
In this section, we scale up the hyperelasticity problem~\cref{eq:hel,eq:lame} in 3D, consisting of millions of DOFs. To isolate pathologies that arise in large-scale problems, we consider a simple geometry of the unit cube $\Omega = [0, 1]^3$ as the computational domain. A clamped zero-displacement boundary condition is imposed on the bottom surface, defined as $\{(x, y, 0) : (x, y) \in [0, 1]^2\}$, while a prescribed displacement boundary condition $u = (0, 0, \tau)$ is enforced on the top surface, $\{(x, y, 1) : (x, y) \in [0, 1]^2\}$. The boundary displacement parameter $\tau$ is varied continuously from $0$ to $2$.

\subsubsection*{Dataset and neural operator training}
To evaluate the scalability and performance of the proposed framework under a large-scale 3D regime, we discretize the physical domain with $N^3 = 64^3$ elements, yielding approximately $0.8\times 10^6$ DOFs. Within this high-dimensional setting, a trajectory-based dataset required to train an FPNO would exceed $600\text{~GB}$ of raw state memory. Furthermore, memory allocation constraints inherent to the Transolver architecture restrict the training batch sizes, rendering the computational cost of an iterative fixed-point optimization phase prohibitively expensive. Consequently, we restrict our large-scale evaluation exclusively to the WS-Newton strategy. 

To construct the parameter-to-state training dataset, we sample $M = 100$ discrete loading 
parameters from a uniform distribution, $\tau \sim \mathcal{U}[0, 2]$. The corresponding reference solutions $u^*(\tau)$ are computed utilizing the IC-Newton solver with 10 steps. This offline data generation phase successfully converges all target residuals below a threshold of $10^{-7}$, requiring $4.22$~hours of wall-clock time distributed across $64$ parallel CPU cores. The neural operator is subsequently optimized to approximate the direct static solution map $\tau \mapsto u^*(\tau)$ with the MSE loss, finally achieving a relative $L^2$ test error of $10^{-3}$ after $42.4$~hours of training on a single NVIDIA H100 GPU.

\subsubsection*{Numerical result}
The performance of the WS-Newton scheme is first tested at an intermediate displacement of $\tau = 0.8$. At this discretization scale, direct solvers exceed available hardware memory allocations (1TB). To bypass this limitation, we employ the CG linear solver preconditioned via an algebraic multigrid (AMG) scheme to execute fast, memory-efficient Newton steps. Under this arrangement, the numerical solver successfully refines the operator's initial inference field, satisfying the strict relative stopping criterion of $\|r_k\|_2 / \|r_0\|_2 \le 10^{-8}$ within 4 iterations. However, comprehensive evaluation across a broader, highly strained spectrum of boundary displacements ($\tau \in \{0.2, 0.4, \dots, 2\}$) demonstrates that a purely data-driven initialization lacks the requisite robustness to solve a large-scale energy minimization problem.

\begin{table}[t]
\centering
\small
\caption{Newton iteration counts for the large-scale robustness evaluation on the $64^3$ grid ($0.8\times 10^6$ DOFs) across varying boundary displacements $\tau$. Convergence is defined by achieving a relative residual norm tolerance of $\|r_k\|_2 /\|r_0\|_2 \le 10^{-8}$ under AMG-preconditioned CG linear solvers. The label ``--'' designates solver divergence or stagnation prior to meeting the target tolerance.}
\label{tab:cube-64-robustness}
\begin{tabular}{c|cccc}
\toprule
$\tau$ & \textbf{PT} & \textbf{PT + NGMRES} & \textbf{FT} & \textbf{FT + NGMRES} \\
\midrule
0.2 & 4 & 4& 5& 4 \\
0.4 & -- & -- & 5& 4 \\
0.6 & -- & 5& 5& 5 \\
0.8 & 4& 4& 6& 5 \\
1.0 & -- & 5& 6& 6 \\
1.2 & -- & -- & 6& 6 \\
1.4 & -- & -- & 7& 7 \\
1.6 & -- & -- & 7& 7 \\
1.8 & -- & -- & 7& 7 \\
2.0 & -- & 5& -- & 8 \\
\bottomrule
\end{tabular}
\end{table}

\Cref{tab:cube-64-robustness} outlines the robustness test result. Here, ``PT'' denotes the baseline pre-trained neural operator optimized solely via macroscopic, data-driven losses. As observed in the first column, the WS-Newton loop accelerated by the unregularized, purely data-driven surrogate successfully converges in only two test cases ($\tau \in \{0.2, 0.8\}$), diverging across all other highly strained states. 

To expand the radius of convergence and filter out localized prediction errors, we inject a nonlinear preconditioning step utilizing a Nonlinear GMRES (NGMRES) solver immediately before entering the linearized system updates. This macro-smoothing step effectively removes localized high-gradient artifacts, enhancing the robustness. While this combined ``PT + NGMRES'' framework improves the success rate from $20\%$ to $50\%$, it remains insufficient to resolve the severe instabilities triggered at highly nonlinear loading states such as $\tau = 1.2$ or $\tau = 1.8$.

To overcome this persistent barrier, we apply the physics-informed fine-tuning strategy (denoted as ``FT'') again, adjusting the underlying neural operator parameters $\theta$ directly against the governing energy. The fine-tuning phase comprises $60$ AdamW optimization steps evaluating the energy across $30$ equispaced $\tau$ values per iteration. This explicit energy minimization dramatically stabilizes the hybrid method, eliminating the majority of previous failure points, with the extreme loading case of $\tau = 2$ remaining the sole instance of failure. Ultimately, combining the fine-tuned neural operator with the nonlinear smoothing loop (``FT + NGMRES'') makes the hybrid method converge to all test cases, demonstrating the viability of the framework for large-scale systems.

\subsubsection*{Numerical result at a finer mesh with super-resolution}
Finally, we evaluate the scalability limits of the proposed hybrid architecture under a multi-million-scale discretization utilizing an $N^3 = 128^3$ mesh. This extreme configuration comprises approximately $6.44 \times 10^6$ DOFs. To circumvent the prohibitive offline computational costs associated with training neural operators directly at this scale, we reuse the neural operator trained on the coarser $N^3 = 64^3$ grid.

We investigate two distinct super-resolution projection strategies: \emph{direct} inference and \emph{interpolation}. The direct inference strategy leverages the mesh-free nature of neural operators, evaluating the operator directly at the $128^3$ nodal coordinate matrix. Conversely, the interpolation strategy executes inference at the native resolution scale ($64^3$) before projecting the fields onto the refined $128^3$ mesh via linear interpolation. 

To resolve the Newton direction from these super-resolved states, we initially implemented the standard AMG preconditioned CG linear solver. However, this combination triggered immediate numerical stagnation. As the spatial mesh size $h$ refines, the condition number of the discrete system Jacobian scales accordingly, drastically magnifying sensitivity to localized field perturbations. Because the underlying hyperelasticity formulation \cref{eq:hel} is highly nonlinear, the localized errors embedded within the super-resolved field projections may break the positive definiteness of the initial discrete Jacobian, as discussed in \cref{sec:spectral_analysis}. To bypass this spectral instability, we restructure the downstream solver within an inexact Newton framework. We substitute the strict inner linear system solver with the FGMRES method, choosing linear tolerance adaptively following~\cite{eisenstat1996choosing}. This variable tolerance configuration serves as an algorithmic low-pass filter, making the outer iterative path robust to super-resolution noise. Equipped with this stabilization trick, the hybrid framework overcomes the stagnation and drives the system to convergence.

\begin{figure}[t]
    \centering
    \includegraphics[width=1\linewidth]{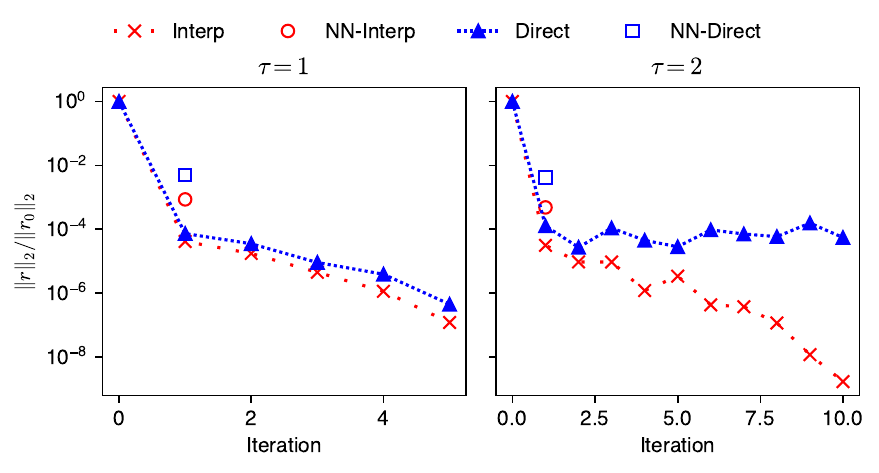}
    \caption{Convergence plots of the relative residual norm $\|r_k\|_2 / \|r_0\|_2$ at \(N^3 = 128^3\) mesh for target loading steps $\tau = 1$ and $\tau = 2$.}
    \label{fig:128cube-residuals}
\end{figure}

\Cref{fig:128cube-residuals} presents convergence plots for two super-resolution strategies evaluated at \(\tau \in \{1, 2\}\). The neural operator-driven initializations immediately reduce the initial relative residual norm by approximately two to three orders of magnitude (circle and square). Beginning from these optimized initial configurations, WS-Newton successfully reduces the residual norm below the threshold for \(\tau = 1\). However, the \emph{direct} strategy does not work for \(\tau = 2\). Because of this reason, we shall only consider the \emph{interpolation} strategy for the following comparison study with IC-Newton. Moreover, we note that the global convergence profiles do not exhibit quadratic rates; this linear-to-superlinear degradation is the expected trade-off when adopting loose, adaptive inner solver tolerances to regularize indefinite Jacobian matrices.

\begin{table}[t]
\centering
\caption{Performance comparison between IC-Newton and the zero-shot super-resolved WS-Newton on the $128^3$ mesh. Values in parentheses indicate the total number of sequential continuation steps. All simulations were benchmarked on 128 cores.}
\label{tab:cube-speedup}
\resizebox{\columnwidth}{!}{
\begin{tabular}{llcccc}
\hline
$\tau$ & \textbf{Method} & $\|r\|_2$ & \textbf{Iterations} & \textbf{Wall Time} (s) & \textbf{Speedup} \\ \hline
\multirow{2}{*}{$1$} & IC-Newton (10 steps) & $1.68 \times 10^{-6}$ & 22 & 538.2 & $1.00$ (baseline) \\
& WS-Newton & $1.45 \times 10^{-7}$ & 7  & 132.9 & $4.05\times$ \\ \hline
\multirow{2}{*}{$2$} & IC-Newton (20 steps) & $2.21 \times 10^{-6}$ & 42 & 1074.0 & $1.00$ (baseline) \\
& WS-Newton & $4.44 \times 10^{-7}$ & 10 & 199.7 & $5.38\times$ \\ \hline
\end{tabular}
}
\end{table}

\Cref{tab:cube-speedup} summarizes the performance comparison between the WS-Newton and the baseline IC-Newton evaluated at loading steps $\tau = 1$ and $\tau = 2$. For the IC-Newton baseline, the number of loading intervals was selected to minimize total computational wall time (1 step per \(\Delta \tau = 0.1\)). The classical incremental technique demands 22 cumulative iterations across 10 steps to resolve the $\tau = 1$ path, increasing up to 42 iterations across 20 steps at $\tau = 2$. By contrast, the WS-Newton framework with interpolation converges within only 7 and 10 iterations, respectively. This reduction translates into a $4.05\times$ wall-clock speedup for \(\tau =1\) and a 
$5.38\times$ acceleration for $\tau = 2$ , demonstrating excellent computational efficiency gains.

\begin{figure}[ht]
    \centering
    \includegraphics[width=1\linewidth]{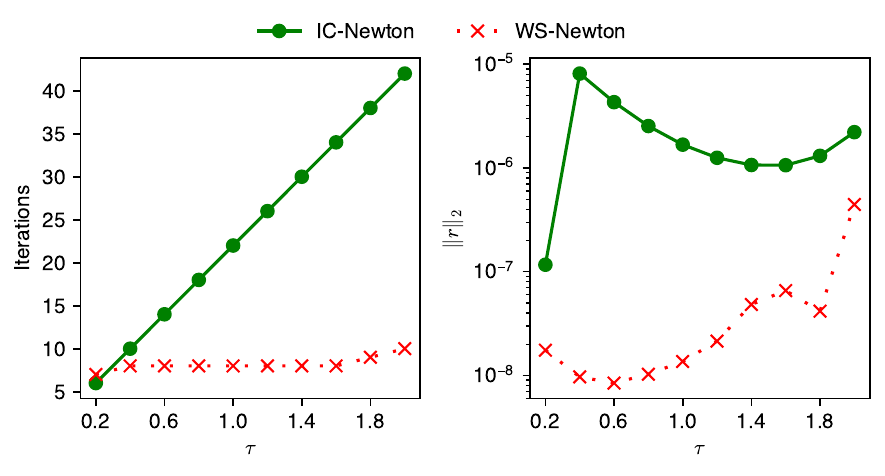}
    \caption{Quantitative performance profiles across a parameter sweep $\tau \in \{0.2, 0.4, \dots 2\}$ at the $128^3$ grid scale. Left: Newton iteration counts till convergence. Right: Final converged relative residual norm magnitudes.}
    \label{fig:128cube-iterations}
\end{figure}

Lastly, \cref{fig:128cube-iterations} tracks solver behavior across a comprehensive parametric sweep of the displacement boundary condition, $\tau \in \{0.2, 0.4, \dots, 2.0\}$. The left panel maps the total Newton iteration counts, and the right panel presents the final relative residual norms, demonstrating the strict preservation of robustness at large-scale dimensions. We note that we intentionally set a stricter convergence criterion for WS-Newton in order to ensure a fair comparison for iteration counts.

\section{Conclusion}
\label{sec:conclusion}
In this work, we presented the numerical failure modes that arise within recently emerging neural operator-based hybrid Newton solvers. We demonstrated that while the NP-Newton method provides better robustness, its training dataset requires storing trajectories, which limits its scalability to multi-million DOF systems. Conversely, while the WS-Newton method exhibits worse robustness when optimized purely via data-driven losses, it does not require a trajectory dataset, hence more scalable. Furthermore, our eigenvalue analysis revealed that errors from neural operator predictions can introduce negative eigenvalues into the discrete system Jacobian matrix, causing Krylov iterative linear solvers to stall or diverge. We showed that a physics-informed fine-tuning phase serves as an efficient remedy, expanding the solver's radius of convergence. Taken together, our findings establish that specialized treatments are mandatory to construct stable, effective hybridizations of machine learning and the Newton method.

Several critical avenues for future research remain open. First, because the practical viability of hybrid solvers depends fundamentally on the offline training phase, developing efficient data collection pipelines, such as utilizing optimal experimental design~\cite{shin2016nonadaptive}, remains important. Second, establishing formal theoretical guarantees is indispensable; while \cref{prop:ws_newton} provides an insight into the convergence of the WS-Newton method, a parallel convergence proof for the NP-Newton loop remains unknown. Finally, transitioning these hybrid methods from academic research benchmarks to widespread industrial deployment demands investigation under more complex and challenging scenarios.

\appendix
\section{Computating environments}
Neural operator predictions and automatic differentiation are handled by PyTorch~\cite{paszke2019pytorch}, while the finite element discretizations and underlying linear algebra systems are handled by FEniCS~\cite{ufl2014, dolfinx2023}, PETSc~\cite{petsc2025}, and Gmsh~\cite{geuzaine2009gmsh}. Figures are generated using Matplotlib~\cite{hunter2007matplotlib} and NumPy~\cite{harris2020array}. MUMPS direct linear solver package~\cite{amestoy2001fully} is used by default.

Apple M3 CPU in \cref{sec:pedagogical_example}. AMD EPYC 9554 64-core processors in the other examples.

\section{Fixed-point neural operator}\label{sec:fpno}
FPNO was first proposed in \cite{lee2025neural}. It consists of backbone \(B^\theta\) and scaling \(S^\theta\) networks. It approximates the solution \(u^*\) by
\[
    u^\theta = u_k + \tanh\left(\|r_k\|_2 S^\theta(r_k / \|r_k\|_2)\right)B^\theta(u_k).
\]
The \(\tanh\left(\|r_k\|_2 S^\theta(r_k / \|r_k\|_2)\right)\) part ensures that the update size is small enough whenever the current iterate \(u_k\) is close to the solution \(u^*\), because \(\|r_k\|_2\) will be small in such a case. It is possible to choose arbitrary neural networks for the backbone and scaling networks.

\section{Details for neural operators}
\Cref{tab:hyperparameter} records architectural details for neural operators employed in \Cref{sec:examples}.
\begin{table}[ht]
\caption{Details for neural operators}\label{tab:hyperparameter}
\resizebox{\columnwidth}{!}{
\begin{tabular}{lcc}\toprule
\textbf{examples} & \textbf{architecture} & \textbf{hyperparameters} \\\midrule
\Cref{sec:pedagogical_example} & MLP & [2, 10, 10, 2], \(\tanh\) \\
\Cref{sec:ldc} & Transolver++ & \(d_h = 128\), 8 heads, 32 slices, \(\text{gelu}\), depth (5, 2) for \((B^\theta, S^\theta)\) \\
\Cref{sec:2dhel} &MLP & The same as \cite{lee2025neural}. \\
\Cref{sec:3dhel} & Transolver++ & \(d_h = 128\), 8 heads, 32 slices, \(\text{gelu}\), depth 5 \\
\bottomrule
\end{tabular}
}
\end{table}

\section*{Acknowledgement}
The first author gratefully acknowledges Ehsan Kharazmi for his valuable comments during the early stages of this work. Part of this research was conducted using the computational resources and services provided by the Center for Computation and Visualization at Brown University. This research is supported by DOE-MMICS SEA-CROGS DE-SC0023191.

\bibliographystyle{assets/siamplain}
\bibliography{library}
\end{document}